\documentclass[smallextended]{amsart}

\usepackage{graphicx}

\usepackage{amsfonts,amssymb,amsmath,bm,amsthm}
\usepackage[hidelinks]{hyperref}

\newtheorem{Theorem}{Theorem}
\newtheorem{Lemma}[Theorem]{Lemma}
\newtheorem{Proposition}[Theorem]{Proposition}
\newtheorem{Question}[Theorem]{Question}
\newtheorem{Corollary}[Theorem]{Corollary}

\theoremstyle{remark}
\newtheorem{Definition}[Theorem]{Definition}
\newtheorem{Remark}[Theorem]{Remark}
\newtheorem{Example}[Theorem]{Example}

\newcommand{\ZZ}{\mathbb Z}
\newcommand{\QQ}{\mathbb Q}

\newcommand{\RR}{\mathbb R}
\newcommand{\TT}{\mathbb T}
\newcommand{\EE}{\mathbb E}
\newcommand{\CC}{\mathbb C}

\newcommand{\FF}{\mathbb F}

\newcommand{\eg}{\emph{e.g.}}

\newcommand\CVD{{\hfill\hfil{\lower 2 pt\hbox{\vrule\vbox to 7pt
{\hrule width 6pt\vfill\hrule}\vrule}}}\vskip 0.5cm}

\newcommand{\pitilde}{\widetilde{\pi}}
\newcommand{\pfrak}{\mathfrak{p}}
\newcommand{\mfrak}{\mathfrak{m}}
\newcommand{\ev}{\operatorname{ev}}
\newcommand{\expC}{\exp_C}
\newcommand{\eC}{e_C}

\begin{document}

\title{On Certain Generating Functions in Positive Characteristic}

\author{F. Pellarin \& R. B. Perkins}

\address{F. Pellarin: Institut Camille Jordan, UMR 5208 Site de Saint-Etienne, 23 rue du Dr. P. Michelon, 42023 Saint-Etienne, France}

\email{federico.pellarin@univ-st-etienne.fr}

\address{R. B. Perkins: IWR, University of Heidelberg, Im Neuenheimer Feld 368, 69120 Heidelberg, Germany}
\email{rudolph.perkins@iwr.uni-heidelberg.de}

\date{\noindent January 2016. \smallskip \\ Part of this research occurred while the first author was supported by the ANR HAMOT. \\ The second author is supported by the Alexander von Humboldt Foundation.}

\maketitle

\begin{abstract}
We present new methods for the study of a class of generating functions introduced by the second author which carry some formal similarities with the Hurwitz zeta function. 
We prove functional identities which establish an explicit connection with certain deformations of the Carlitz logarithm introduced by M. Papanikolas and involve the Anderson-Thakur function and the Carlitz exponential function. 
They collect certain functional identities in families for a new class of $L$-functions introduced by the first author. 
This paper also deals with specializations at roots of unity of these generating functions, producing a link with Gauss-Thakur sums.
\end{abstract}

\section{Introduction}

Let $A$ be the ring $\FF_q[\theta]$ of polynomials in an indeterminate $\theta$ with coefficients in the finite field of $q$ elements $\FF_q$, and let $A^+ \subseteq A$ denote the subset of monic elements in $\theta$. We denote by $K$
the fraction field of $A$ and by $K_\infty := \FF_q((1/\theta))$ the completion of $K$ at
the infinite place. We also denote by $\CC_\infty$ the completion of
an algebraic closure $K_\infty^{ac}$ of $K_\infty$; we endow it with
the norm $|\cdot|$ determined by $|\theta|=q$.

Let $\mathbb{T}_s$ be the standard Tate algebra in the indeterminates $t_1,\dots,t_s$ with
coefficients in $\CC_\infty$. Explicitly, it is the ring of formal series
\[\sum_{i_1,\dots,i_s \geq 0} c_{i_1,\dots,i_s} t_1^{i_1} \cdots t_s^{i_s} \in \CC_\infty[[t_1,\dots,t_s]]\]
such that $c_{i_1,\dots,i_s} \rightarrow 0$ as $i_1+\cdots + i_s \rightarrow \infty$. It can also be viewed as the completion of $\CC_\infty[t_1,\dots,t_s]$ for
the Gauss norm associated to the absolute value $|\cdot|$ of $\CC_\infty$; see \S \ref{andergenfunct} for the precise definition. 
The Tate algebra $\TT_s$ is endowed with the open and continuous $\FF_q[t_1,\dots,t_s]$-linear automorphism
$\tau:\TT_s \rightarrow \TT_s$ defined by $\tau(c)=c^{q}$, for
$c \in \CC_\infty$. For each $i \in \Sigma_s := \{1,2,\dots,s\}$, let $\chi_{t_i} : A \rightarrow \FF_q[t_i] \subseteq \TT_s$ be the $\FF_q$-algebra morphism determined by $\theta \mapsto t_i$, and note that much of what follows depends on these choices of algebra generators. Finally, we denote by $\EE_s$ the sub-algebra of $\TT_s$ whose elements are entire in the all variables $t_1,\dots,t_s$. When $s = 1$, we will write $t_1 = t, \TT_1 = \TT$ and $\EE_1 = \EE$.

The purpose of this note is to introduce new methods into the study of the properties of the $\TT_s$-valued functions
\[\psi_s(z) := \sum_{a\in A}\frac{\chi_{t_1}(a)\cdots\chi_{t_s}(a)}{z-a} = \sideset{}{'}\sum_{a\in A}\frac{\chi_{t_1}(a)\cdots\chi_{t_s}(a)}{z-a};\]
here (and elsewhere in this article) the primed summation indicates the exclusion of $0 \in A$ from the sum, and second equality follows from the fact that $\chi_t(0) = 0$. For all positive integers $s$, the series $\psi_s$ converges in $\TT_s$ for all $z \in \CC_\infty \setminus (A \setminus \{0\})$, and, in particular, $\psi_s$ converges in an open neighborhood of $z = 0$. These functions are a positive characteristic analog of the Hurwitz zeta function and were introduced by the second author as generating functions, when $s$ is fixed and $n \equiv s \pmod{q-1}$, for the following multivariate $L$-functions
\[L(\chi_{t_1}\cdots\chi_{t_s},n) := \sum_{a\in A^+}a^{-n}\chi_{t_1}(a)\cdots\chi_{t_s}(a);\]
see Lemma \ref{Lgenlem} below. For example, when $s = 1$, for all $z \in \CC_\infty$ such that $|z|<1$, we have $$\psi_1(z) = \sum_{k \geq 0} z^{k(q-1)} L(\chi_t,1+k(q-1)).$$

The $L$-functions above were introduced by the first author in \cite{PEL} and converge for all positive integers $n$ and $s$ to non-zero elements of $\TT_s(K_\infty)$, the standard Tate algebra in $t_1,\ldots,t_s$ with coefficients in the local field $K_\infty$. They may be regarded as complementary objects to the global $L$-functions in equal positive characteristic introduced by Goss (see \cite{PEL1}).

The authors plan to use the functions $\psi_s$ in a crucial way in a forth coming work developing the theory of $\TT_s$-valued vectorial modular forms introduced by the first author in \cite{PEL}. Indeed, these functions allow us to introduce a notion of regularity at infinity for weak $\TT_s$-valued vectorial modular forms on the Drinfeld upper-half plane which guarantees the finite rank of the $\TT$-modules of such forms of a given weight and type. Additionally, we think that the functions $\psi_s$ carry their own intrinsic interest, and we intend to demonstrate this by cataloging some of their finer properties in the sequel.

As a product of our labor, we find a new proof of a basic relation that exists between the $L$-series value $L(\chi_t,1)$ evaluated in $t$ at roots of unity and Gauss-Thakur sums. The precise statement is given in Theorem \ref{angpelomeg} below.

\subsection{Main Functions}
Let $d_0 = 1$ and, for all $i \geq 1$, let $d_i := \prod_{j = 0}^{i-1} (\theta^{q^i} - \theta^{q^j})$.
Define
\[\expC(z) := \sum_{i \geq 0} d_i^{-1} z^{q^i} \ \text{ and } \ \pitilde := -\lambda_\theta^q \prod_{i \geq 1} (1 - \theta^{1-q^i})^{-1} \in K_\infty(\lambda_\theta), \] respectively, the \emph{Carlitz exponential function} and a \emph{fundamental period of the Carlitz exponential}, the latter being unique up to the choice of $(q-1)$-th root $\lambda_\theta$ of $-\theta$ in the algebraic closure $K^{ac}$ of $K$. The Carlitz exponential defines an $\FF_q$-linear endomorphism of $\CC_\infty$, and its kernel is $\pitilde A$. Finally, set $\eC(z) := \expC(\pitilde z)$. This function is entire, $A$-periodic, and has as its only zeroes the elements of $A$ with multiplicity one.

We consider the following special case of an {\em Anderson generating function for the Carlitz module} studied by the first author in \cite{PEL0},
\[f_{t}(z) := \sum_{j \geq 0} \eC(z \theta^{-j-1}) t^j = \sum_{n\geq 0}\frac{(\pitilde z)^{q^n}}{(\theta^{q^n} - t)d_n} \in \TT,\]
where the equality was first established; see also \cite{EGP}.
From the equality above we see that $f_t$ is an entire function of $z$ and a meromorphic function of $t$ with simple poles at $t = \theta, \theta^q, \dots$. From the definition, we see that, for each $z \in \CC_\infty$, $f_t(z)$ satisfies the $\tau$-difference equation
\begin{equation} \label{AGFdiffeq}
\tau(f_t(z)) = \eC(z) + (t - \theta)f_t(z).\end{equation}
In particular, if $z=1$, we obtain $f_t(1) = \omega(t)$, the \emph{Anderson-Thakur function}, which satisfies 
\begin{equation}\label{taudiffomega} \tau(\omega) = (t-\theta)\omega, \end{equation} 
as $\eC(1) = 0$, and generates the rank one $\FF_q[t]$-submodule of $\TT$ of solutions $X \in \TT$ of the equation $\tau(X) = (t - \theta)X$. With this, one readily shows that
\[\omega(t)=\lambda_\theta\prod_{i\geq 0}\left(1-\frac{t}{\theta^{q^i}}\right)^{-1}\in\TT,\]
as shown in \cite[Lemma 2.5.4]{AT}; see also \cite[Lemma 3.1.4]{GAtm}.
The function $\omega$ is in many ways a cousin of Euler's gamma function, see \emph{e.g.} the discussion in \cite{PEL}.

The difference equation \eqref{AGFdiffeq} also allows for a connection with the following function introduced by Papanikolas in \cite{PAP}
\[\bm{L}_\alpha(t) := \alpha + \sum_{j \geq 1}\frac{\alpha^{q^j}}{(t - \theta^q)(t - \theta^{q^2})\cdots(t - \theta^{q^j})}.\]
This function converges in $\TT$ for each $\alpha \in \CC_\infty$ such that $|\alpha| < q^{\frac{q}{q-1}}$. The function $\bm{L}_\alpha(t)$ is a deformation of the Carlitz logarithm function; indeed, one may evaluate the variable $t$ at $\theta$ to recover the Carlitz logarithm $\log_C(\alpha) = \bm{L}_{\alpha}(\theta)$ in this way. The function $\bm{L}_\alpha(t)$ plays a crucial role in Papanikolas' proof of a ``folklore conjecture:'' \emph{a $K$-linearly independent set of Carlitz logarithms of non-zero elements of $K^{ac}$ is algebraically independent over $K$}. Indeed, $\bm{L}_\alpha(t)$ is the main ingredient in the construction of a rigid analytic trivialization of Papanikolas' \emph{logarithm motive}. The reader may consult \cite{PAP} for the exact details of its use.

The following connection made by El-Guindy and Papanikolas \cite{EGP} between $\bm{L}_{\alpha}(t)$ and $f_t(z)$ will be of special interest. Let $z \in \CC_\infty$ be such that $|z| < 1$. Letting $\alpha = \eC(z)$, one has an identity of elements of $\TT$
\begin{equation} \label{papdiffeq}
\bm{L}_{\alpha}(t) = (\theta-t)f_t(z);\end{equation}
note that the right side above converges for all $z \in \CC_\infty$.

\subsubsection{A zeta function realization of Papanikolas' functions}
The next result was first obtained in \cite{PER} by the second author using interpolation polynomials and $\tau$-difference formalism, and we shall give a new proof by using the ultrametric maximum principle. We wish to stress here the connection with the function of Papanikolas described above. The proof of the following result appears in \S \ref{proofofthe1}.

\begin{Theorem}\label{thrudy1} Fix $z \in \CC_\infty$ such that $|z| < 1$, and let $\alpha = \eC(z)$. The following identity holds in $\TT$,
\begin{equation}\label{identityrudy}
\eC(z)\psi_1(z) = \frac{\widetilde{\pi}\bm{L}_{\alpha}(t)}{(\theta-t)\omega(t)}.
\end{equation}
\end{Theorem}

As the title of this section suggests, the identity of the previous theorem gives a zeta function realization of the function $\bm{L}_\alpha$ akin to the identity
\begin{equation}\label{pelsid}
L(\chi_t,1) = \frac{\pitilde}{(\theta-t)\omega(t)},\end{equation}
discovered by the first author in \cite{PEL}, which gave a zeta function realization for the rigid analytic trivialization 
\[\Omega := \tau(\omega)^{-1} = \frac{1}{(t-\theta)\omega}\] 
of the dual Carlitz motive. 

\subsection{Main Results}

\subsubsection{Functional identities for $\psi_s$ with $s$ arbitrary}
Without the restriction on $s$ necessary in \cite{PER}, we shall prove a qualitative generalization of the explicit identities of the second author \cite[Theorem 1.1]{PER}, and, in particular, Theorem \ref{thrudy1}; see \S \ref{proofofthe1}. Clearly, the following result may also be stated in terms of Papanikolas' functions, but we leave the explicit details to the reader.

To follow, any empty product shall be understood to equal $1$, and for a finite set $I$, we denote its cardinality by $|I|$. We will prove the next result in \S \ref{pfthm2}.

\begin{Theorem}\label{pellarinperkins} Let $s$ be a positive integer.
For each subset $I \subset \Sigma_{s} = \{1,2,\dots,s\}$, $I\neq\Sigma_{s}$, there exists a polynomial $g_I \in K(t_1,\dots,t_s)[X]$ of degree at most $\frac{|\Sigma_s \setminus I|}{q}$ in $X$, 
such that $g_I(e_C)\prod_{j \in \Sigma_s \setminus I} \omega(t_j)^{-1} \in \EE_s[\eC]$ and, for all $z \in \CC_\infty$,
\[ \frac{1}{\pitilde}{\omega(t_1) \cdots \omega(t_s) \eC(z) \psi_s(z)} = \prod_{i \in \Sigma_{s}} f_{t_i}(z) + \eC(z)\sum_{I \subsetneq \Sigma_{s}} g_I(\eC(z)) \prod_{i \in I} f_{t_i}(z).\]
\end{Theorem}

We are unable to explicitly describe the polynomials $g_I$ guaranteed by Theorem \ref{pellarinperkins} above; we do however give an explicit description of the {\it functions} $g_I(\eC)$ in the proof of Theorem \ref{pellarinperkins} below. 
In contrast with the explicit results of the second author \cite[Theorem 1.1]{PER}, we are led to pose the following question.
\begin{Question}
Are the polynomials $g_I(X)$ of Theorem \ref{pellarinperkins} actually elements of \newline $ K(t_1,\dots,t_s)$?
\end{Question}

The entireness, in all variables, of the function $g_I(e_C)/\prod_{j \in \Sigma_s \setminus I} \omega(t_j) \in \EE_s[\eC]$ tells us that the poles of the coefficients of $g_I(X)$ can only occur at $t_j = \theta^{q^i}$, for $j \in \Sigma_s \setminus I$ and $i \geq 0$. Conceivably, with the methods of \cite{ANG&PEL1}, we can give an upper bound on the powers of $\theta$ which can occur, but we do not pursue such questions in this article.

Further, we do not yet know if there is a connection of the identity of Theorem \ref{pellarinperkins} above with the 
theory of $t$-motives. Rather, our primary interest in proving the theorem above is that it demonstrates how the functional identities of \cite[Theorem 1]{ANG&PEL1} assemble in families when $s$ is fixed and $\alpha$ varies in the class of $s \pmod{q-1}$; see Lemma \ref{Lgenlem}. 
We shall also discuss the growth of the functions $\psi_s$ and $f_t$ as $|z|_{\Im} := \inf_{\kappa \in K_\infty}|z-\kappa| \rightarrow \infty$; clearly, $\psi_s$ tends toward zero, and by Remark \ref{chigrwthinfty}, $f_t$ is unbounded. Thus the identity of Theorem \ref{pellarinperkins} above demonstrates a rich interplay between the functions $\eC, f_{t_1},\dots,f_{t_s}$ when $|z|_\Im$ is large.

Finally, we must mention a paper of Gekeler \cite{Gekcrelles}, wherein he uses functions akin to $\psi_1$ to establish surjectivity of the de Rham morphism for the cohomology of Drinfeld modules. Given a positive integer $h$, a lattice $\Lambda \subseteq \CC_\infty$ with exponential function $e_\Lambda$, and an $A$-linear function $\mathcal{X}: \Lambda \rightarrow \CC_\infty$, Gekeler considers the {\it quasi-periodic functions}
\[F_{h,\mathcal{X}}(z) := \sum_{\lambda \in \Lambda} \mathcal{X}(\lambda) (e_\Lambda(z)/(z - \lambda))^{q^h},\]
which can be shown to converge uniformly for all $z \in \CC_\infty$ to $\FF_q$-linear functions. Further, they interpolate the values of $\mathcal{X}$ on $\Lambda$ and satisfy, for each $a \in A$, the functional identities
\[F_{h,\mathcal{X}}(az) - aF_{h,\mathcal{X}}(z) = \eta_a(e_\Lambda(z))\]
for some $\eta_a$ in the ring of twisted polynomials $\tau \CC_\infty\{\tau\}$, giving the functions $A$-linearity upon restriction of $z$ to $\Lambda$. Thus, perhaps Theorem \ref{pellarinperkins} above suggests an extension of Gekeler's results for the multilinear pointwise products $z \mapsto \mathcal{X}_1(z) \mathcal{X}_2(z) \cdots \mathcal{X}_s(z)$ of the $A$-linear functions $\mathcal{X}_1, \mathcal{X}_2, \dots, \mathcal{X}_s : \Lambda \rightarrow \CC_\infty$. We hope to return to this question in the future.

\subsubsection{Evaluation at roots of unity} \label{evalmap}

Another basic feature of these multivariate $L$-series, and hence the functions $\psi_s$, is the possibility of evaluating the variables $t_1, \dots, t_s$ at roots of unity; see $\eg$ \cite{ANG&PEL1,ANG&PEL2,APT,PEL}.
As a highlight of such considerations, detailed below, we recover the following connection, due originally to Angl\`es and the first author \cite{ANG&PEL2}, between the Anderson-Thakur function $\omega$ and Thakur's Gauss sums, from \cite{THA},
\[g(\chi) := \sum_{a \in (A / \pfrak A)^\times} \chi(a)^{-1} \eC(a/\pfrak),\]
where $\chi(a)$ denotes the image of $a$ under the $\FF_q$-algebra map $\chi$ determined by $\theta \mapsto \zeta$ for some fixed root $\zeta$ of $\pfrak$. The Gauss-Thakur sum $g(\chi)$ is non-zero for all such characters $\chi$.

Letting $\ell_0:=1$ and $\ell_i:=(\theta-\theta^{q^i})\ell_{i-1} \in A$, for all positive integers $i$, we have the following result whose proof appears in \S \ref{Mpolyprops} below.

\begin{Theorem} \label{angpelomeg}
For all $\zeta \in \FF_q^{ac}$ with minimal polynomial in $A$ of degree $d$, we have
 \[\omega|_{t = \zeta} = -\chi(\ell_{d-1}) g(\chi).\]
\end{Theorem}

It is worth noting that the sign of $g(\chi)$ is quickly obtained from the previous theorem, and the absolute value of $g(\chi)$ is readily seen from the definition above and will be clear from Cor. \ref{exactcoeffs} below. See \cite{THAsign} where such questions are handled in full generality.

In the case of more indeterminates, we let $\mathfrak{p}_1,\ldots,\mathfrak{p}_s$ be primes (that is, irreducible monic polynomials) of
$A$; we also set $\mfrak=\mathfrak{p}_1\cdots \mathfrak{p}_s$. Let us choose, for all $i=1,\ldots,s$, a root $\zeta_i$ of $\mathfrak{p}_i$ in $\FF_q^{ac}$, the algebraic closure of $\FF_q$ in $\CC_\infty$.
Let $$\ev_\mfrak:\TT_s\rightarrow\CC_\infty$$ be the evaluation map
sending a formal series $\sum_{i_1,\ldots,i_s}c_{i_1,\ldots,i_s}t_{1}^{i_1}\cdots t_{s}^{i_s}\in\TT_s$ to $\sum_{i_1,\ldots,i_s}c_{i_1,\ldots,i_s}\zeta_{1}^{i_1}\cdots \zeta_{s}^{i_s}$. We emphasize the dependence on $\zeta_1,\ldots,\zeta_s$, rather than on $\pfrak_1, \dots, \pfrak_s$.

One easily checks that, for each $a \in A$, the functions $\psi_s$ satisfy the following simple transformation rule,
\[\psi_s(z+a) = \sum_{I \subseteq \Sigma_s} \sum_{b \in A} (z-b)^{-1} \prod_{i \in I} \chi_{t_i}(b) \prod_{j \in \Sigma_s \setminus I} \chi_{t_j}(a).\]
Now, while the $\psi_s$ themselves are not $A$-periodic, for any such evaluation map $\ev_\mfrak$, we have
\[\ev_\mfrak(\psi_s(z+\mfrak a))=\ev_\mfrak(\psi_s(z)),\quad \text{ for all } a\in A.\]
Furthermore, letting
\begin{equation} \label{umfrakdefid} u_\mfrak(z) := \frac{1}{\pitilde} \sum_{a \in A} \frac{1}{z-\mfrak a}=\frac{1}{\mfrak\eC(z/\mfrak)},\end{equation}
Proposition \ref{evchiprop} below gives
$$\pitilde^{-1}\ev_\mfrak(\psi_s)\in K^{ac}(u_\mfrak)\cap K^{ac}[[u_\mfrak]],$$
and convergence of the formal series holds for $z\in\CC_\infty$ such that $|z|_\Im$
is big enough. Explicitly computing the rational functions above
appears to be a difficult task. However, we can say something more when $s = 1$.

To state the final result we would like to highlight, we define, \emph{ad hoc}, $g(\chi^{-1}) := (-1)^d \pfrak / g(\chi)$, which is is the Gauss-Thakur sum for the character $\chi^{-1}$ and is an element of $A[\zeta,\eC(1/\pfrak)]$ by \cite[Prop. 15.2]{ANG&PEL1}. For more precision, in the case of a single prime $\pfrak$ with root $\zeta$ we now write $\ev_\zeta$ for the map $\ev_\pfrak$ described above. 

\begin{Theorem} \label{degcoeffthm}
When $|z|_\Im$ is sufficiently large, we have 
\[\pfrak \ev_\zeta(\psi_1/\pitilde) = u_\pfrak^{{|\pfrak|(q-1)}/{q}} \sum_{i \geq 0} a_i u_\pfrak^i \in A[\zeta,\eC(\pfrak^{-1})][[u_\pfrak]],\] 
with $a_0 = (-1)^{d+1}g(\chi^{-1})\chi(\ell_{d-1})^{-1}$.
\end{Theorem}

The proof of Theorem \ref{degcoeffthm} appears in \S \ref{Mpolyprops} below.

\subsection*{Acknowledgements}
Both authors wish to thank D. Goss for his continued interest in our project and helpful comments on previous versions of this work. We especially thank the anonymous referees for their careful reading and numerous helpful comments. It was their keen eyes which led us to discover a more satisfactory and correct proof of Theorem \ref{pellarinperkins}. 

\section{Entire functions with values in Banach algebras}\label{entirefunctions}

Let $\mathcal{R}$ be a $\CC_\infty$-Banach algebra with norm $|\cdot|_{\mathcal{R}}$ extending the norm $|\cdot|$ of $\CC_\infty$, i.e. such that for all $z \in \CC_\infty$ we have $|z\cdot 1|_\mathcal{R} = |z|$; see \cite[Chapter 13]{SCH} for basic definitions and results. We identify $\CC_\infty$ with a subalgebra of $\mathcal{R}$ via $\CC_\infty \cdot 1 \subset \mathcal{R}$. Throughout this section, we assume that $|\mathcal{R}|_{\mathcal{R}}=|\CC_\infty|=q^\QQ\cup\{0\}$
(\footnote{Where $q^\QQ=\{x\in\RR_{>0};x^n=q^m,\text{ for some }m,n\in\ZZ\setminus\{0\}\}$.}), that is, the set of norms
$|f|_{\mathcal{R}}$ for $f\in \mathcal{R}$ is equal to the set of norms $|x|$, $x\in\CC_\infty$.

\begin{Definition} 
A function
$f:\CC_\infty\rightarrow\mathcal{R}$
is called $\mathcal{R}$-\emph{entire} (or simply \emph{entire}) if, on every bounded subset $B$ of $\CC_\infty$,
$f$ can be obtained as a uniform limit of polynomial functions $f_i\in\mathcal{R}[z]$,
$f_i:B\rightarrow \mathcal{R}$.
\end{Definition} 

In particular, for such an $\mathcal{R}$-entire function $f$ and for all $z\in\CC_\infty$, we can write 
$$f(z)=\sum_{i\geq 0}c_iz^i,\quad c_i\in\mathcal{R},$$
and $\lim_{i\rightarrow \infty}|c_i|_{\mathcal{R}}^{1/i}=0$. The identification of $\CC_\infty$ with $\CC_\infty \cdot 1 \subset \mathcal{R}$ allows every $\CC_\infty$-entire function to be identified with an $\mathcal{R}$-entire function.

\begin{Remark}
Similarly, a natural notion of  a {\em $\mathcal{R}$-rigid analytic function} exists but will be marginal for the purposes of the present paper.
\end{Remark}

\subsection{$\mathcal{R}$-entire functions}

\begin{Proposition}\label{entire}
A bounded $\mathcal{R}$-entire function is constant.
\end{Proposition}
\begin{proof}
We follow Schikhof's \cite[Theorems 42.2 and 42.6]{SCH}. We denote by $B_0(1)$ the
set of the $z\in\CC_\infty$ such that $|z|<1$. We note that if
$$f:B_0(1)\rightarrow\mathcal{R}$$ is defined by
$f(x)=\sum_{i\geq 0}a_ix^i$ for $a_i\in\mathcal{R}$, then
\begin{equation}\label{maxmod1}
\sup\{|f(x)|_{\mathcal{R}};|x|\leq 1\}=\max\{|a_n|_{\mathcal{R}};n\geq0\}\end{equation}
(cf. loc. cit. Lemma 42.1). Already, taking $x=1$, we see that
$\max\{|a_n|_{\mathcal{R}};n\geq0\}$ is well defined and equals the supremum
of the same set. We may assume that this maximum is one by rescaling $f$ by
an element $c\in\CC_\infty^\times$ (because of our assumption that
$|\mathcal{R}|_{\mathcal{R}}=|\CC_\infty|$).
We then have:
$$\sup\{|f(x)|_{\mathcal{R}};|x|< 1\}\leq \sup\{|f(x)|_{\mathcal{R}};|x|\leq1\}\leq \max\{|a_n|_{\mathcal{R}};n\geq0\}=1$$ and we need to show that we have equalities everywhere.
If $|a_0|_{\mathcal{R}}=1$, we have $|f(0)|_{\mathcal{R}}=1$ and we are
done. Otherwise, let $N$ be the smallest integer $j$ such that $|a_j|_{\mathcal{R}}=1$. We have that $N>0$. Let us choose $\epsilon>0$
such that $\epsilon<1-\max\{|a_0|_{\mathcal{R}},|a_1|_{\mathcal{R}},\ldots,|a_{N-1}|_{\mathcal{R}}\}$. Since $|\CC_\infty^\times|=q^\QQ$ is dense in the positive real numbers, there exists $x\in\CC_\infty^\times$ such that
$1-\epsilon<|x^N|<1$. Then, $|f(x)|_{\mathcal{R}}=|x^N|\geq 1-\epsilon$.

Now let $f$ be $\mathcal{R}$-entire and let us consider $r\in q^\QQ$.
We deduce from (\ref{maxmod1}) with $f$ replaced by $f(ax)$ with $|a|=r$ that
$$\sup\{|f(x)|_{\mathcal{R}};|x|\leq r\}=\max\{|a_n|_{\mathcal{R}}r^n;n\geq0\}.$$
If there exists $M\in\RR_{\geq0}$ such that $|f(x)|_{\mathcal{R}}\leq M$
for all $x\in\CC_\infty$, then, for all $r$ as above, $\max\{|a_n|_{\mathcal{R}}r^n;n\geq0\}\leq M$ which implies $a_1=a_2=\cdots=0$ and $f(x)=a_0$
for all $x\in\CC_\infty$. 
\end{proof}

We choose $r\in q^\QQ$. We denote by $B_0(r)$ the set $\{z\in\CC_\infty;|z|<r\}$. A function
$f:B_0(r)\rightarrow\mathcal{R}$ is said {\em $\mathcal{R}$-analytic (on $B_0(r)$)}
if there exist $c_0,c_1,\ldots\in\mathcal{R}$ such that $|c_i|_{\mathcal{R}}r^i\rightarrow\infty$
as $i\rightarrow\infty$, and such that, for all $z\in B_0(r)$, we have the equality
$f(z)=\sum_{i\geq 0}c_iz^i$.

\begin{Proposition}\label{uniformclosedness} 
Let $f_1,f_2,\ldots$ and $f$  be functions $B_0(r)\rightarrow\mathcal{R}$.
Let us assume that $f_1,f_2,\ldots$ are $\mathcal{R}$-analytic on $B_0(r)$ and that
$f=\lim_{i\rightarrow\infty}f_i$ uniformly on $B_0(r)$. Then, $f$ is $\mathcal{R}$-analytic 
on $B_0(r)$. In particular, if $f_1,f_2,\ldots$ and $f$ are functions
$\CC_\infty\rightarrow\mathcal{R}$ such that $f_1,f_2,\ldots$ are $\mathcal{R}$-entire and
such that, for all $r>0$, $f=\lim_{i\rightarrow\infty}f_i$ uniformly on $B_0(r)$,
then $f$ is $\mathcal{R}$-entire.
\end{Proposition}
\begin{proof}
This follows from an easy adaptation of the uniform closedness principle \cite[Theorem 42.2, (ii)]{SCH}
that we leave to the reader. 
\end{proof}

\subsection{$A$-periodic $\mathcal{R}$-entire functions}

We need a few preliminary lemmas.

\begin{Lemma}\label{lemmaQ1}

Let $z$ be in $\CC_\infty$. Then, one of the following two properties holds.
\begin{enumerate}
\item There exists $a\in A$ such that $|z-a|<1$.
\item $|z|_{\Im}\geq1$.
\end{enumerate}

\end{Lemma}

\begin{proof}
Assume that for all $a\in A$, $|z-a|\geq 1$, and let $\alpha$ be an element of $K_\infty$.
Then, we can write (uniquely) $\alpha=a+\beta$ with $a\in A$ and $\beta\in \frac{1}{\theta}\FF_q[[\frac{1}{\theta}]]$. We observe that $|\beta|<1$. Hence, $|\beta|\neq|z-a|$ and 
$$|z-\alpha|=|z-a-\beta|=\max\{|z-a|,|\beta|\}\geq1.$$ In particular, if the property 1 of the Lemma
does not hold, then, we have that $|z|_{\Im}\geq1$.
\end{proof}

\begin{Lemma}\label{lemmaQ2}

Let $z$ be an element of $\CC_\infty$.

\begin{enumerate}
\item If there exists $a\in A$ such that $|z-a|<1$, then
$|e_C(z)|<|z-a|q^{\frac{q}{q-1}}$.
\item There exists $\kappa_0>1$, independent of $z$, with the following property. If $|z|_{\Im}\geq 1$, 
then $q^{-\kappa_0}\leq |e_C(z)|\leq q^{-1}$.
\end{enumerate}

\end{Lemma}

\begin{proof}
If there exists $a\in A$ such that $|z-a|<1$, the conclusion is easily obtained
as it is well known from its Newton polygon that $\exp_C:\CC_\infty\rightarrow\CC_\infty$
induces an isometric $\FF_q$-automorphism of the disk $\{z'\in\CC_\infty;|z'|<|\widetilde{\pi}|\}$; recall that $|\widetilde{\pi}|=q^{\frac{q}{q-1}}$.

The second case of the lemma follows easily from \cite[Proposition 9]{GekelerCICMA}. 
\end{proof}

Let $\mathcal{P}^!(\mathcal{R})$ be the $\mathcal{R}$-algebra of
$\mathcal{R}$-entire functions which are $A$-periodic, that is, the $\mathcal{R}$-algebra of
$\mathcal{R}$-entire functions $f$ such that $f(z+a)=f(z)$ for all $z\in\CC_\infty$ and $a\in A$.

The following result will 
be used in the proof of Theorem \ref{pellarinperkins} below. We continue with the assumption 
$|\mathcal{R}|_{\mathcal{R}}=|\CC_\infty|$.

\begin{Proposition}\label{temperedper}
Let $f\in\mathcal{P}^!(\mathcal{R})$ be such that there exist an integer $N$ and a real number $C>0$ with
$|f(z)|_\mathcal{R}\leq C\max\{1,|\eC(z)|^N\}$ for all $z\in\CC_\infty$. 
Then, $f$ is a polynomial of $\mathcal{R}[\eC]$ of degree at most $N$.
\end{Proposition}

\begin{proof} We recall that
the function $\eC:\CC_\infty\rightarrow\mathcal{R}$, is $\mathcal{R}$-entire,
$A$-periodic, and vanishes at $z=a$ with order one for all $a\in A$ and has no other zeroes. This implies that there exist elements $c_0,\ldots,c_{N}$ of $\mathcal{R}$ such that the function $g:\CC_\infty\rightarrow\mathcal{R}$
$$g=\frac{f-\sum_{i=0}^{N}c_i\eC^{i}}{\eC^{N}}$$
is $\mathcal{R}$-entire and $A$-periodic.

We set $C_1=\sup_{|z|<1}|g(z)|$. Let $z\in\CC_\infty$. If there exists 
$a\in A$ such that $|z-a|<1$, then $|g(z)|_{\mathcal{R}}<1$. Otherwise, 
Lemma \ref{lemmaQ1}
implies that $|z|_{\Im}\geq 1$. From Lemma \ref{lemmaQ2},
we then deduce
$$|g(z)|_{\mathcal{R}}\leq q^{c_0N}\max\{1,q^{-N},|c_0|_{\mathcal{R}},|c_0|_{\mathcal{R}}q^{-1},\ldots,|c_N|_{\mathcal{R}}q^{-N}\}=:C_2.$$
For all $z\in\CC_\infty$, $|g(z)|_{\mathcal{R}}\leq\max\{C_1,C_2\}$ and $|g|_{\mathcal{R}}$
is bounded over $\CC_\infty$, hence constant by Proposition \ref{entire}. But $g(0)=0$
and $g=0$. 
\end{proof}

\begin{Definition}
An $\mathcal{R}$-entire function
$f:\CC_\infty\rightarrow\mathcal{R}$ is called {\em tempered} if there exists a real number $C>0$ and an integer
$N\geq 0$
depending on $f$ such that for all $z\in\CC_\infty$, $|f(z)|_\mathcal{R}\leq C\max\{1,|\eC(z)|\}^N$.

We denote by $\mathcal{T}(\mathcal{R})$ the $\mathcal{R}$-algebra of tempered $\mathcal{R}$-entire functions. We also set $\mathcal{P}(\mathcal{R})=\mathcal{P}^!(\mathcal{R})\cap\mathcal{T}(\mathcal{R})$.
\end{Definition}

\begin{Remark}{ 
Note that by Proposition \ref{temperedper}, if $|\mathcal{R}|_{\mathcal{R}}=|\CC_\infty|$,
then $$\mathcal{P}(\mathcal{R})=\mathcal{R}[e_C].$$}
\end{Remark}

\section{Anderson generating functions} \label{andergenfunct}

For the remainder of this paper, $\mathcal{R}$ will either be $\TT_s$ or $\EE_s$ equipped with the Gauss norm $|\cdot|_{\mathcal{R}}=\|\cdot\|$, given by $\|\sum_i c_i t^i\| = \max_i\{|c_i|\}$. Both of these rings satisfy $|\mathcal{R}|_\mathcal{R}=|\CC_\infty|$.

We observe that the function $z \mapsto \omega(t)^{-1} f_t(z)$
 extends the map $\chi_t : A \rightarrow \FF_q[t] \subseteq \TT$ of the introduction to a $\TT$-entire function $\CC_\infty\rightarrow\TT$, see \cite[Th. 2.19]{PER}. Thus, to follow we shall write
 \[\chi_t(z) := \omega(t)^{-1} f_t(z).\]
 From the definition of $f_t(z)$, it is easy to see that $\chi_t(z)$ is $\FF_q$-linear.
  
For all $z \in \CC_\infty \setminus A$ we also set
\[u(z) := \frac{1}{\pitilde} \sum_{a \in A} \frac{1}{z-a} = \frac{1}{\expC(\widetilde{\pi}z)} = \frac{1}{\eC(z)}.\]

\begin{Lemma}\label{chilem}
The map $\chi_{t}:\CC_\infty\rightarrow\EE$ is the unique $\FF_q$-linear, $\EE$-entire function which restricts on $A$ to the $\FF_q$-algebra isomorphism $A\rightarrow\FF_q[t]$ determined by the assignment $\theta \mapsto t$, and satisfies
\[ \|\chi_t(z)\|\leq\max\{1,|\eC(z)|^{\frac{1}{q}} \} \quad \text{ for all } z \in \CC_\infty.\]
\end{Lemma}

\begin{proof}
That $\chi_t(z)$ is $\FF_q$-linear, has image in $\EE$ and is $\EE$-entire, follows easily from the definitions. 

From (\ref{AGFdiffeq}) and (\ref{taudiffomega}) we deduce that $\chi_t(z)$ satisfies the $\tau$-difference equation
\begin{equation}\label{taudiff1}
\tau(\chi_t(z))=\chi_t(z) + (u(z)(t-\theta)\omega(t))^{-1}.\end{equation}

Thus, if $\|\chi_t(z)\| > 1$, then $\|\tau(\chi_t(z))\| = \|\chi_t(z)\|^q > \|\chi_t(z)\|$, and from \eqref{taudiff1} we obtain
\[\|\chi_t(z)\| = \|(t - \theta) \omega(t)\|^{-\frac{1}{q}} |\eC(z)|^{\frac{1}{q}}.\]
In all cases, we have
\[\|\chi_t(z)\| \leq \max\{1, \|(t - \theta) \omega(t)\|^{-\frac{1}{q}} |\eC(z)|^{\frac{1}{q}}\} \leq \max\{1, |\eC(z)|^{\frac{1}{q}}\},\]
yielding our inequality. Now, let $g:\CC_\infty\rightarrow\EE$ be another $\EE$-entire, $\FF_q$-linear function,
interpolating the map $\chi_t$ over $A$. Then, the function $f(z)=\chi_t(z)-g(z)$ 
is $\EE$-entire, $A$-periodic, vanishing over $A$, and such that
for all $z\in\CC_\infty$, $\|f(z)\|\leq\max\{1,|\eC(z)|^{q^{-1}}\}$. Then, $f$ is constant, therefore zero because 
vanishing at the elements of $A$. 
\end{proof}

\begin{Remark}\label{chigrwthinfty}{
When $|z|_\Im$ is sufficiently large, it is possible to show that $\|\chi_t(z)\| > 1$, and we obtain the equality $$\|\chi_t(z)\| = \|(t - \theta) \omega(t)\|^{-\frac{1}{q}} |\eC(z)|^{\frac{1}{q}}.$$}
\end{Remark}

We will also employ the following property. 

\begin{Lemma}\label{isometrylemma}
If $z\in\CC_\infty$, $|z|<q$, then $\|\chi_t(z)\|=|z|$.
In particular, $\chi_t$ induces an injective $\FF_q$-linear map
$\{z : |z|<q\}\rightarrow\{f\in\EE : \|f\|<q\}$.
\end{Lemma}

\begin{proof}
By definition, the Newton polygon of the function $f_t(z)$ (as a series of $\TT[[z]]$ and with the 
Gauss norm over $\TT$)
is equal to the Newton polygon of $e_C(\frac{z}{\theta})$. The Lemma follows
from the identity $\|\frac{\pitilde}{\theta\omega}\|=1$. 
\end{proof}

We mention a final interesting result concerning the function $\chi_t(z)$ though we do not need it in the sequel.
\begin{Lemma}
For $z\in\CC_\infty$, we have $\chi_t(z)=0$  if and only if $z=0$.
\end{Lemma}

\begin{proof}
We denote by $\boldsymbol{e}$ the unique $\FF_q[t]$-linear endomorphism of $\TT$
extending the Carlitz exponential function $\exp_C:\CC_\infty\rightarrow\CC_\infty$. 
Following \cite{ANG&PEL2},
this yields a short exact sequence of $A[t]$-modules
$$0\rightarrow \widetilde{\pi}A[t]\rightarrow\TT\rightarrow C(\TT)\rightarrow0,$$
where $C(\TT)$ is the unique $A[t]$-module structure over $\TT$ extending  $\FF_q[t]$-linearly the Carlitz module $\theta \mapsto \theta + \tau$
over $\CC_\infty$. If we set $\boldsymbol{f}(h)=\boldsymbol{e}(\widetilde{\pi}h)$,
then the kernel of $\boldsymbol{f}$ is $A[t]$ and 
$$\chi_t(z)=\frac{\boldsymbol{f}\left(\frac{z}{\theta-t}\right)}{\boldsymbol{f}\left(\frac{1}{\theta-t}\right)},\quad z\in\CC_\infty,$$
from which we easily deduce that $\chi_t(z)=0$ if and only if $z=0$, if $z\in\CC_\infty$. 
\end{proof}

\begin{Remark}
{ The function $z\mapsto \frac{\chi_t(z)}{z}$ is a non-constant $\EE$-entire function
$\CC_\infty\rightarrow\EE$ which never vanishes.}
\end{Remark}

\section{The functions $\psi_s$ and the proof of Theorem \ref{pellarinperkins}}\label{sect4}

Before giving the proof of Theorem \ref{pellarinperkins}, we require three auxiliary results.

\begin{Lemma} \label{Lgenlem} 
Let $z\in\CC_\infty$ be such that $|z|<1$. Let $s$ be a positive integer, and let  $m$ be the unique integer $0 < m \leq q-1$ such that $m\equiv s\pmod{q-1}$. As power series in $\frac{(\pitilde z)^m}{\omega(t_1)\cdots\omega(t_s)}K(t_1,\dots,t_s)[[\pitilde z]]$, we have
$$ z\psi_s(z)=\sum_{k\geq 0}(\pitilde z)^{m+k(q-1)}\frac{L(\chi_{t_1}\cdots\chi_{t_s},m+k(q-1))}{\pitilde^{m+k(q-1)}}.$$
\end{Lemma}
\begin{proof}
We have $\psi_s(z)=\sum_{a\in A\setminus\{0\}}\frac{\chi_{t_1}(a)\cdots\chi_{t_s}(a)}{z-a}$.
Since for $a\in A\setminus\{0\}$ and $|z| < 1$, we have $\frac{1}{z-a}=-\frac{1}{a}\sum_{i\geq 0}\left(\frac{z}{a}\right)^i$, we deduce
\begin{eqnarray*}
\psi_s(z)&=&-\sum_{a\in A\setminus\{0\}}a^{-1}\chi_{t_1}(a)\cdots\chi_{t_s}(a)\sum_{i\geq 0}\left(\frac{z}{a}\right)^i\\
&=&-\sum_{i\geq 0}z^i\sum_{\lambda\in\FF_q^\times}\lambda^{s-i-1}\sum_{a\in A^+}\frac{\chi_{t_1}(a)\cdots\chi_{t_s}(a)}{a^{i+1}}\\
&=&\sum_{k\geq 0}z^{m-1+k(q-1)}L(\chi_{t_1}\cdots\chi_{t_s},m+k(q-1)).
\end{eqnarray*}
Multiplying by $z$ and applying \cite[Theorem 1]{ANG&PEL1} finishes the proof. 
\end{proof}

\begin{Remark}
We observe that for $s=0$, the previous lemma reduces to the well-known identity of Carlitz
$z\sum_{a\in A} (z-a)^{-1} = \sum_{k\geq 0}z^{k(q-1)}\zeta_C(k(q-1))$, where, for positive integers $n$, $\zeta_C(n) := \sum_{a\in A^+}a^{-n}$ are the Carlitz zeta values.\end{Remark}

\begin{Lemma}\label{algebraicindependence}
The functions $\eC,\chi_{t_1},\ldots,\chi_{t_s}$ are algebraically
independent over the field $\CC_\infty((t_1,\ldots,t_s))$.
\end{Lemma}
\begin{proof} We set $\mathcal{K}=\CC_\infty((t_1,\ldots,t_s))$ and $\mathcal{L}=\mathcal{K}((e_C))$. We choose the canonical embedding of $\TT_s \rightarrow \mathcal{K}$. 

We claim that $\chi_{t_1}$,..., $\chi_{t_s}$ are algebraically independent over $\mathcal{L}$.
Let us suppose, by contradiction, that 
there exists $P\in\mathcal{L}[X_1,\ldots,X_s]$, irreducible, such that
$P(\chi_{t_1},\ldots,\chi_{t_s})=0$, let $\mathcal{P}$ be the prime ideal of $\mathcal{L}[X_1,\ldots,X_s]$
generated by $P$. Then, for all $a\in A$, $$P_a=P(X_1+\chi_{t_1}(a),\ldots,X_s+\chi_{t_s}(a))\in \mathcal{P},$$ 
because $\chi_t(z+a)=\chi_t(z)+\chi_t(a)$, for all $z\in\CC_\infty$ and $a\in A$.

Hence, $P_a=HP$ for some $H\in\mathcal{L}[X_1,\ldots,X_s]\setminus\{0\}$ ($P_a$ is clearly non-zero).
The total degrees of $P$ and $P_a$ for $a\in A$ are equal; then $H\in\mathcal{L}^\times$.
The constant term of $P_a$ is the polynomial $P(\chi_{t_1}(a),\ldots,\chi_{t_s}(a))$.
If $R$ is the constant term of $P$, we thus have the identity:
$$P(\chi_{t_1}(a),\ldots,\chi_{t_s}(a))=HR,\quad a\in A.$$
But the image of the map
$A\rightarrow\mathbb{A}^s(\mathcal{L})$
defined by $a\mapsto(\chi_{t_1}(a),\ldots,\chi_{t_s}(a))$
is Zariski-dense. Hence, $P$ is non-zero and constant, a contradiction.

In particular, the functions $\chi_{t_1}(z),\ldots,\chi_{t_s}(z)$ are algebraically independent over 
$\mathcal{K}(e_C)$ which is a transcendental extension of $\mathcal{K}$. 
\end{proof}

\begin{Remark}{
In fact, the functions of the previous lemma are even algebraically independent over $\CC_\infty((t_1,\ldots,t_s))(z)$, but we do not need this property.} \end{Remark} 

\begin{Proposition}\label{elementary} Let $K$ be a field,
let us consider formal series
$f_0,f_1,\ldots,f_s\in K[[Z]]$, let $L/K$ be a field extension and let us suppose that the series $f_i$ are algebraically independent
over $L$. Let $P$ be a polynomial of $L[X_0,\ldots,X_s]$ such that
$P(f_0,\ldots,f_s)\in K[[Z]]$. Then, $P\in K[X_0,\ldots,X_s]$.
\end{Proposition}
\begin{proof} Let $P\in L[X_0,\ldots,X_s]$ be such that
$P(f_0,\ldots,f_s)\in K[[Z]]$, and let
$d$ be its degree.
Since the functions $f_0,\ldots,f_s$ are algebraically independent
over $L$, the monic monomials $M_0,\ldots,M_N$ in $f_0,\ldots,f_s$
of degree $\leq d$ are linearly independent over $L$ and there
is an $L$-rational linear combination $F=\sum_ia_iM_i\in K[[Z]]$.
The matrix $U$ whose columns are the coefficients of the formal series $M_i \in K[[Z]]$
(it has $N+1$ columns and infinitely many rows) has maximal
rank and the right multiplication by the column matrix with coefficients $a_i$
yields an infinite column of elements of $K$ which cannot be identically
zero. Extracting from $U$ a non-singular square matrix of order $N+1$
and inverting it, we deduce that the coefficients $a_0,\ldots,a_N$ are
in $K$. 
\end{proof}

\subsection{Proof of Theorem \ref{pellarinperkins}} \label{pfthm2}
Recall that we have set $\Sigma_{s} := \{1,2,\dots,s\}$. For subsets $I \subseteq \Sigma_{s}$ we define $\chi^I(z) := \prod_{i \in I} \chi_{t_i}(z)$, and, as before, $\chi^\emptyset(z) = 1$. For $I \subseteq \Sigma_s$, we recall that $|I|$ denotes the cardinality of $I$. 

We have
\begin{eqnarray*}
\psi_s(z) &=& \sum_{a \in A} \frac{\chi_1(a-z+z)\cdots \chi_s(a-z +z)}{z-a} \\ &=& \widetilde{\pi} u(z) \chi^{\Sigma_s}(z) + \sum_{I \subset \Sigma_s} \chi^I(z) \sum_{a \in A}\frac{\chi^{I^c}(a-z)}{z-a},\end{eqnarray*}
where $I \subset \Sigma_s$ runs over all strict subsets of $\Sigma_s$, and $I^c := \Sigma_s \setminus I$.

We argue that the sums $h_I(z) := \sum_{a \in A}\frac{\chi^{I^c}(a-z)}{z-a}$, for strict subsets $I \subset \Sigma_s$, define $\EE_s$-entire, $A$-periodic, tempered functions. Let $r \in q^\QQ$, and let $z \in B_0(r) : = \{z \in \CC_\infty : |z|<r\}$. 
Consider $a \in A$. If $|z - a| \geq q$, by Lemma \ref{chilem}, we have 
\begin{equation} \label{rationorm}
\left\|\frac{\chi^{I^c}(a-z)}{z-a} \right\| \leq \frac{\max\{1, \sup_{z \in B_0(r)}|\eC(z)|^{\frac{|I^c|}{q}} \}}{|z-a|},\end{equation} 
and the right side tends to zero as $|a| \rightarrow \infty$. If $|z-a| < q$, by Lemma \ref{isometrylemma}, \begin{equation} \label{rationorm2}
\left\|\frac{\chi^{I^c}(a-z)}{z-a} \right\| = |z-a|^{|I^c|-1}.\end{equation} 
The singularity of $\frac{\chi^{I^c}(z-a)}{z-a}$ at $z = a$ is clearly removable, and thus this function is $\EE_s$-entire for all $a \in A$.  It follows from Proposition \ref{uniformclosedness} that $h_I$ is $\EE_s$-entire, and clearly $h_I$ is $A$-periodic. 
Temperedness follows from \eqref{rationorm} and \eqref{rationorm2} giving, for some positive real number $C$,
\[\left\|\sum_{a \in A}\frac{\chi^{I^c}(a-z)}{z-a} \right\| \leq C\max\{1, |e_C(z)|^{\frac{|I^c|}{q}}\}.\]
Thus, for all $I \subset \Sigma_s$, $h_I \in \EE_s[e_C]$ of degree in $e_C$ at most $|I^c|/q$. 

For $I \subsetneq \Sigma_s$, let $g_I \in \TT_s[Z]$ be such that $g_I(\eC) = \pitilde^{-1} h_I\prod_{j \in I^c}\omega(t_j)$. It remains to show that 
$g_I(\eC) \in K(t_1,\dots,t_s)[\eC]$.
Multiplying out the numerators of the functions $h_I$, this follows easily from Lemma \ref{Lgenlem} and Proposition \ref{elementary} applied to the functions $\eC$, $f_{t_1}$, $\ldots$, $f_{t_s}$ which
are algebraically independent by Lemma \ref{algebraicindependence} and have their expansions in powers of $\widetilde{\pi}z$ defined over $K(t_1,\ldots,t_s)$. \hfill $\qed$

\subsection{Proof of Theorem \ref{thrudy1}} \label{proofofthe1}
Now we may deduce Theorem \ref{thrudy1} from Theorem \ref{pellarinperkins} and Lemma \ref{chilem}. By Theorem \ref{pellarinperkins} we have
\[\psi_1 = \pitilde u \chi_{t} + h_\emptyset,\]
for some $h_\emptyset \in \TT$. By Lemma \ref{chilem}, we see that $\psi_1 - \pitilde u \chi_{t} \rightarrow 0$ as $|z|_\Im \rightarrow \infty$. Hence, $h_\emptyset = 0$. We make the connection with Papanikolas' function via \eqref{papdiffeq}, and this puts the constraint on $z$. \hfill $\qed$

\section{Evaluations at roots of unity}
Throughout this section, $s$ will be a fixed positive integer, and we retain the notation $\Sigma_{s} = \{1,2,\dots,s\}$. We recall our notation from the introduction: we let $\mathfrak{p}_1,\ldots,\mathfrak{p}_s$ be primes (that is, irreducible monic polynomials) of
$A$; we also set $\mfrak=\mathfrak{p}_1\cdots \mathfrak{p}_s$. Let us choose, for all $i=1,\ldots,s$, a root $\zeta_i$ of $\mathfrak{p}_i$ in $\FF_q^{ac}$, the algebraic closure of $\FF_q$ in $\CC_\infty$.
Let $$\ev_\mfrak:\TT_s\rightarrow\CC_\infty$$ be the evaluation map
sending a formal series $\sum_{i_1,\ldots,i_s}c_{i_1,\ldots,i_s}t_{1}^{i_1}\cdots t_{s}^{i_s}\in\TT_s$ to $\sum_{i_1,\ldots,i_s}c_{i_1,\ldots,i_s}\zeta_{1}^{i_1}\cdots \zeta_{s}^{i_s}$.

For $j \geq 0$, let $A(j)$ be the $\FF_q$-vector subspace of $A$ consisting of all polynomials of degree strictly less than $j$; of course, $A(0) = \{0\}$. Let $Z$ be an indeterminate over $\CC_\infty$, and let $K\{Z\}$ denote the non-commutative ring of $\FF_q$-linear polynomials with coefficients in $K$, with composition as product. For $j \geq 0$, let $E_j(Z) := \frac{1}{d_j}\prod_{a \in A(j)} (Z-a) \in K\{Z\}$; see \cite[Theorem 3.1.5]{GOS}. For $a \in A$, let $C_a(Z) := \sum_{j \geq 0} E_j(a) Z^{q^j}$. These are the polynomials arising from the $\FF_q$-linear map $A \rightarrow A\{Z\}$ determined by $\theta \mapsto \theta Z + Z^q$, by \cite[Corollaries 3.5.3 and 3.5.4]{GOS}.

For each subset $J \subset \Sigma_{s}$ define
\[ M^J_\mfrak(Z) := \sum_{b \in A/ \mfrak A} \frac{C_\mfrak(Z - \eC(b / \mfrak))}{(Z - \eC(b / \mfrak))} \prod_{j \in J} b(\zeta_j) \in A[\zeta_{1},\dots,\zeta_s,\eC(\mfrak^{-1})][Z]. \]
Observe that the degree in $Z$ of $M_\mfrak^J(Z)$ is strictly less than $|\mfrak|$ and that $M^J_\mfrak$ is the Lagrange interpolation polynomial for the data $\eC(a / \mfrak) \mapsto \mfrak \prod_{j \in J} a(\zeta_j)$, defined for $a \in A/\mfrak A$. For $J=\Sigma_s$ we simply write $M^{\Sigma_s}_\mfrak=M_\mfrak$.

\begin{Proposition} \label{evchiprop}
For all $z \in \CC_\infty \setminus A$, square-free monic polynomials $\mfrak$ and subsets $J\subseteq \Sigma_{s}$, the following identity holds,
\[ {\pitilde^{-1}}\ev_\mfrak({\psi_J(z)}) = \frac{1}{\mfrak} \frac{M_\mfrak^J(\eC(z / \mfrak))}{C_\mfrak(\eC(z/\mfrak))} \in K(\zeta_1,\dots,\zeta_s,\eC(\mfrak^{-1}))(u_\mfrak(z)),\]
and when $|z|_\Im$ is sufficiently large, we have \[\pitilde^{-1}\ev_\mfrak(\psi_J)(z) \in {u_\mfrak}(z) A[\zeta_1,\dots,\zeta_s,\eC(\mfrak^{-1})][[u_\mfrak(z)]].\]
\end{Proposition}

\begin{proof}
We have
\begin{eqnarray}
\nonumber \ev_\mfrak(\psi_J)(z) &=& \sum_{b \in A} (z-b)^{-1} \prod_{j \in J} b(\zeta_j) \\
&=& \nonumber \sum_{b \in A / \mfrak A} \prod_{j \in J} b(\zeta_j) \sum_{a \in \mfrak A} (z - b - a )^{-1} \\
&=& \nonumber \pitilde \sum_{b \in A / \mfrak A} u_\mfrak(z-b) \prod_{j \in J} b(\zeta_j) \\
&=& \label{eqlastline} \frac{\pitilde}{\mfrak} \sum_{b \in A / \mfrak A} \frac{\prod_{j \in J} b(\zeta_j)}{ \eC(z / \mfrak) - \eC(b / \mfrak) }.
\end{eqnarray}
From the third equality to the fourth, we have used the identity \eqref{umfrakdefid},
and the final identity comes from direct comparison with the definition of $M_\mfrak^J$ combined with the observation that $C_\mfrak(Z - \eC(b / \mfrak)) = C_\mfrak(Z),$ for all $b \in A / \mfrak A$.

Returning to \eqref{eqlastline}, and assuming that $|z|_\Im$ is sufficiently large so that, for all $b \in A/\mfrak A$ we have $|\mfrak e_C(b/\mfrak)u_\mfrak| < 1$, the second claim follows from the computation 
\begin{eqnarray}
\nonumber \ev_\mfrak(\psi_J/\pitilde) &=& \sum_{b \in A / \mfrak A} \frac{ u_\mfrak\prod_{j \in J} b(\zeta_j)}{ 1 - \mfrak \eC(\frac{b}{\mfrak}) u_\mfrak } \\
\label{evpsiuexpn} &=&  \sum_{k \geq 0} u_\mfrak^{k+1} \sum_{b \in A/\mfrak A} (\mfrak \eC(\frac{b}{\mfrak}))^k \prod_{j \in J} b(\zeta_j). \end{eqnarray}
\end{proof}

\begin{Corollary} \label{chievalcor} We suppose that $s=1$.
For every prime $\mfrak=\mathfrak{p}$ and all $z \in \CC_\infty$, we have
\[\ev_{\pfrak}(\chi_t(z)) = \pfrak^{-1} M_\pfrak(\eC(z / \pfrak)).\]
\end{Corollary}

\begin{proof}
This is immediate from Theorem \ref{thrudy1} and the previous proposition.
\end{proof}

One would like to describe the precise order of vanishing in $u_\mfrak$ of the evaluations $\ev_{\mfrak}(\psi_s)$ for all square-free conductors $\mfrak$ and positive integers $s$, or equivalently to determine the degree in $Z$ of $M_\mfrak(Z)$. This appears to be difficult in general, but it is possible in the case of $s=1$. We tackle this problem in the next section.

We quickly mention that we may give an upper bound on the order of vanishing of $\ev_\pfrak(\psi_1)$ analytically using Remark \ref{chigrwthinfty} and Theorem \ref{thrudy1}. Starting from the identity $\psi_1=\widetilde{\pi}u\chi_t$, for all $z$ such that $|z|_{\Im}$ is big enough, we have
$\|\psi_1\|=C|u|^{1-1/q}$, for some positive real constant $C$. Hence, using the characterization of the Gauss norm $\|f\| = \sup_{|t| \leq 1}|f(t)|$,
for $\pfrak$ a prime, we obtain $|\ev_{\pfrak}(\psi_1)| \leq C|u_\pfrak|^{\frac{|\pfrak|}{q}(q-1)}$, for a new constant $C$ provided that $|z|_{\Im}$ is big enough. We now know that $\ev_{\pfrak}(\psi_1)\in\CC_\infty[[u_\pfrak]]$, and it follows that
$\ev_{\pfrak}(\psi_1) \in u_\pfrak^{\frac{|\pfrak|}{q}(q-1)}\CC_\infty[[u_\pfrak]].$

\subsection{Gauss-Thakur sums and the polynomial $M_\pfrak$} \label{Mpolyprops}
We now focus our attention completely on the interpolation polynomials $M_\pfrak$ in order to prove Theorems \ref{angpelomeg} and \ref{degcoeffthm} of the introduction.

Let $\pfrak$ be a monic irreducible of degree $d$, and let $\lambda = \eC(\pfrak^{-1})$ be a fixed element of Carlitz $\pfrak$-torsion. Choose a root $\zeta$ of $\pfrak$, and let $\chi$ be the Teichm\"uller character determined by $\theta \mapsto \zeta$. 

The following Gauss sum was first introduced by Thakur in \cite{THA} and already appeared in the introduction,
\[g(\chi) := \sideset{}{'}\sum_{a \in A(d)} \chi(a)^{-1} C_a(\lambda).\]
The sum $g(\chi)$ is a non-zero element of $A[\zeta,\lambda]$. We recall setting $g(\chi^{-1}) := (-1)^d \pfrak / g(\chi)$, which is the Gauss-Thakur sum for the character $\chi^{-1}$ and is an element of $A[\zeta,\lambda]$ by \cite[Prop. 15.2]{ANG&PEL1}. 
We also recall that for each $a \in (A/\pfrak A)^\times$, there is a unique element $\sigma_a$ of the Galois group of $K(\zeta, \lambda)$ over $K(\zeta)$ satisfying $\sigma_a(\lambda) = C_a(\lambda) = \eC(a / \pfrak)$ and that $g(\chi)$ satisfies
\begin{equation}\label{basicGTeq}
\sigma_a(g(\chi)) = a(\zeta) g(\chi)
\end{equation} for all $a \in (A/\pfrak A)^\times$. 

\begin{Lemma}\label{propdegree}
The following identity holds in $g(\chi^{-1}) A[\zeta][Z]$,
\begin{equation} \label{Mchicoeffseq}
M_\pfrak(Z) = (-1)^d g(\chi^{-1}) \sum_{j = 0}^{d-1} Z^{q^j} \sum_{a \in A(d) \setminus A(j) } a(\zeta)^{-1} E_j(a). \end{equation}
\end{Lemma}
\begin{proof}
Define the coefficients $a_i \in A[\zeta]$ via $g(\chi) := \sum_{i = 0}^{|\pfrak|-2} a_i \lambda^i$; one easily sees here that only those indices $i \equiv 1 \pmod{q-1}$ can appear. We observe immediately from \eqref{basicGTeq} that
the polynomial $g(\chi)^{-1}\sum_{i \geq 0} a_i Z^i$ takes the value $a(\zeta)$ for $Z = C_a(\lambda)$ for all $a \in A(d)$ and takes the value zero at $Z = 0$. Further its degree is strictly bounded above by $|A(d)| = |\pfrak|$. Since $\pfrak^{-1} M_\pfrak(Z)$ is the Lagrange interpolation polynomial of the data $C_a(\lambda) \mapsto a(\zeta)$ for $a \in A(d)$, we must have $\pfrak^{-1} M_\pfrak(Z) = g(\chi)^{-1}\sum_{i = 0}^{|\pfrak|-2} a_i Z^i$, by consideration on the degrees in $Z$ and uniqueness of the Lagrange interpolating polynomial. Replacing $g(\chi)^{-1}$ with $(-1)^d \pfrak^{-1} g(\chi^{-1})$, we obtain $M_\pfrak(Z) = (-1)^d g(\chi^{-1}) \sum_{i = 0}^{|\pfrak|-2} a_i Z^i$.

Now we derive an expression for the coefficients $a_i$ appearing above. From the definition of $g(\chi)$, we immediately obtain
\[ g(\chi) = \sideset{}{'}\sum_{a \in A(d) } a(\zeta)^{-1} \sum_{j = 0}^{d-1}  E_j(a) \lambda^{q^j} = \sum_{j = 0}^{d-1} \lambda^{q^j} \sideset{}{'}\sum_{a \in A(d) } a(\zeta)^{-1} E_j(a). \] 
\end{proof}

\begin{Remark}
{ It is an interesting question to determine the $\FF_q$-vector subspace of $\CC_\infty$ of roots to the polynomial $M_\pfrak(Z)$.}
\end{Remark}

We can give a more compact expression for the coefficients of $Z$ and $Z^{|\pfrak|/q}$ in $M_\pfrak$. We require a lemma, and some additional notation. 

Let $x$ and $y$ be indeterminates over $\FF_q$ and write $a(x) \in \FF_q[x]$, resp. $a(y) \in \FF_q[y]$, for the image of $a \in A$ under $\theta \mapsto x$, resp. $\theta \mapsto y$. We recall that for $j \geq 1$, we have set $\ell_j := (\theta  - \theta^{q^j})(\theta - \theta^{q^{j-1}})\cdots(\theta - \theta^q) \in A$ and $\ell_0 := 1$.

\begin{Lemma} \label{telescope}
For all integers $d \geq 1$, we have
\[\sum_{j = 0}^{d-1} \ell_j(x)^{-1} \prod_{k = 0}^{j-1} (y - x^{q^k}) = \ell_{d-1}(x)^{-1} \prod_{j = 1}^{d-1} (y - x^{q^j}).\]
\end{Lemma}
\begin{proof}
We prove this claim by induction on $d$. When $d = 1$, we have two empty products, both equal to $1$ by convention. For $d = 2$, we have 
\[1 + \frac{(y - x)}{x - x^q} = \frac{(y - x^q)}{x - x^q}, \]
as required. Now assume the result holds for some $d \geq 2$. For such $d$ we have 
\begin{eqnarray*} \sum_{j = 0}^{d} \ell_j(x)^{-1} \prod_{k = 0}^{j-1} (y - x^{q^k}) &=&  \frac{1}{\ell_d(x)} \prod_{k = 0}^{d-1} (y - x^{q^k}) + \sum_{j = 0}^{d-1} \frac{1}{\ell_j(x)} \prod_{k = 0}^{j-1} (y - x^{q^k}) \\
 &=&  \frac{1}{\ell_d(x)} \prod_{k = 0}^{d-1} (y - x^{q^k}) + \frac{1}{\ell_{d-1}(x)} \prod_{j = 1}^{d-1} (y - x^{q^j}) \\
 &=& \frac{1}{\ell_{d-1}(x)}(\frac{y-x}{x-x^{q^d}} + 1)\prod_{j = 1}^{d-1} (y - x^{q^j}) \\
 &=& \frac{1}{\ell_d(x)}\prod_{j = 1}^{d} (y - x^{q^j}). 
\end{eqnarray*}
Hence, the result holds for $d+1$ as well and proves the claim. 
\end{proof}

\begin{Corollary} \label{exactcoeffs}
The coefficient of $Z^{|\pfrak|/q}$ in $M_\pfrak$ equals $(-1)^{d+1} g(\chi^{-1})\chi(\ell_{d-1})^{-1}$, and the coefficient of $Z$ equals $(-1)^{d+1} \pfrak \chi(\ell_{d-1})^{-1} (\theta - \zeta)^{-1} g(\chi^{-1})$; both are elements of $A[\zeta,\lambda]$.
\end{Corollary}

\begin{proof}
For non-negative integers $j$, let $A^+(j)$ denote the set of all elements of $A^+$ of degree $j$ in $\theta$.

From \eqref{Mchicoeffseq}, the coefficient of $Z^{|\pfrak|/q}$ equals
\[(-1)^{d+1} g(\chi^{-1}) \sum_{a \in A^+(d-1)} a(\zeta)^{-1} = (-1)^{d+1} g(\chi^{-1}) \chi(\ell_{d-1})^{-1};\]
the latter equality follows from the work of Carlitz \cite[(9.09)]{Car35}.

Similarly, the coefficient of $Z$ equals
\[(-1)^{d+1} g(\chi^{-1}) \sum_{j = 0}^{d-1} \sum_{a \in A^+(j)} a(\zeta)^{-1} a.\]
One may easily calculate that 
\[\sum_{a \in A^+(j)}  \frac{a(y)}{a(x)} = \frac{1}{\ell_{j}(x)} \prod_{k = 0}^{j-1}(y - x^{q^k}),\] 
by comparing zeros of the left and right hand sides above using \cite[(9.15)]{Car35} and the next displayed equation therein and comparing the coefficients of $y^j$, which is again given by \cite[(9.09)]{Car35}. 

Summing over $0 \leq j \leq d-1$, applying Lemma \ref{telescope}, and replacing $x$ by $\zeta$ and $y$ by $\theta$, we obtain the expression above for the coefficient of $Z$. 
\end{proof}

\subsubsection{Proof of Theorem \ref{angpelomeg}}
It suffices to compare the coefficients of $z$ on both the left and right sides of the identity of Cor. \ref{chievalcor}. On the left side we have $\pitilde\ev_\zeta((\theta - t)\omega(t))^{-1}$, while on the right we have $(-1)^{d+1} \pitilde \pfrak^{-1} (\theta - \zeta)^{-1} \chi(\ell_{d-1})^{-1} g(\chi^{-1})$, by Cor. \ref{exactcoeffs}. Hence, using the definition of $g(\chi^{-1}) = (-1)^d \pfrak g(\chi)^{-1}$, we obtain $\ev_\zeta(\omega) = -\chi(\ell_{d-1}) g(\chi)$, as desired. \qed

\subsubsection{Proof of Theorem \ref{degcoeffthm}}
This is an immediate consequence of Prop. \ref{evchiprop} and Cor. \ref{exactcoeffs} just above. \qed \medskip

\begin{Remark}
Let $\pfrak$, $d$ and $\chi$ be as in Theorem \ref{degcoeffthm}. From \eqref{evpsiuexpn} and Theorem \ref{degcoeffthm}, one may deduce the vanishing of the sums $\sum_{b \in (A/\pfrak A)^\times} \eC(b/\pfrak)^k \chi(b)$, for $k = 0,1,,\dots,\frac{|\pfrak|(q-1)}{q}-2$, and the identity
\[ g(\chi^{-1})  = (-1)^{d+1}\pfrak^{\frac{|\pfrak|(q-1)}{q}} \chi(\ell_{d-1}) \sum_{b \in (A/\pfrak A)^\times } \eC(b/\pfrak)^{\frac{|\pfrak|(q-1)}{q}-1} \chi(b).\]
We leave the easy details to the reader. 
\end{Remark}

\subsection{Questions in closing} We conclude with a speculative subsection. Let us fix an integer $s\geq 1$.
As Theorem \ref{pellarinperkins} shows, the functions $\psi_s$ are elements of the $\TT_s$-algebra $\mathcal{R}_s := \TT_s[u,u^{-1}, \{\chi^J\}_{J \subseteq \Sigma_{s}}]$. It follows from Corollary \ref{chievalcor} that for all monic $\mfrak=\pfrak_1\cdots\pfrak_s$, as above, and all $h \in \mathcal{R}_s$, we have $ev_\mfrak(h) \in \CC_\infty(u_\mfrak)$ for every choice of roots $\zeta_i$ of $\pfrak_i$, as we saw for $\psi_s$.

We define the sub-$\TT_s$-algebra $\mathcal{R}_s^0 \subseteq \mathcal{R}_s$ whose
elements are the $h \in \mathcal{R}_s$ such that $\ev_\mfrak(h) \in \CC_\infty[[u_\mfrak]]$ for all $\pfrak_1,\ldots,\pfrak_s$ and all roots $\zeta_i$
of $\pfrak_i$. As proven in Proposition \ref{evchiprop}, we have $\psi_J \in \mathcal{R}_s^0$, for all $J \subseteq \Sigma_{s}$.

We denote by $\TT_{s,i}$ the completion of $\CC_\infty[t_1,\ldots,\widehat{t_i},\ldots,t_s]\subset\TT_s$ for the Gauss norm (the hat denotes an omitted term
and the completion is canonically identified to a subring of $\TT_s$).
We also denote by $\varphi_i$ the $\TT_{s,i}$-linear
endomorphism of $\mathbb{T}_s$ uniquely defined by $\varphi_i(t_i)=t_i^q$.
\begin{Lemma} The ring $\mathcal{R}_s^0$ is $\varphi_i$-stable for all $i$.
That is, for all $f\in\mathcal{R}_s^0$ and all $i=1,\ldots,s$, we have
$\varphi_i(f)\in \mathcal{R}_s^0$.
\end{Lemma}

\begin{proof} We know that for any choice of $\zeta_1,\ldots,\zeta_s\in\FF_q^{ac}$, for $f\in\mathcal{R}_s^0$, we have $f|_{t_i=\zeta_i}\in\CC_\infty[[u_\mfrak]]$, where $\pfrak_i$ is the prime of $A$ vanishing at $\zeta_i
$ (for $i=1,\ldots,s$) and $\mfrak=\pfrak_1\cdots\pfrak_s$.
Let us write $g=\varphi_j(f)$ for a fixed $j\in\{1,\ldots,s\}$.
Since $\zeta_j^q$ is conjugate to $\zeta_j$ over $\FF_q$, we again get, thanks to
the hypothesis on $f$, that $g|_{t_i=\zeta_i}\in\CC_\infty[[u_\mfrak]]$.
This means that $g\in\mathcal{R}_s^0$. 
\end{proof}

\begin{Example} We set $s=1$, case in which we write $t=t_1$, $\TT=\TT_1$
and $\varphi=\varphi_1$. Observe that
$\varphi(f)=(\tau^{-1}(f))^q$. From this and  \eqref{AGFdiffeq}, \eqref{taudiffomega}, \eqref{papdiffeq}, and \eqref{identityrudy} we deduce
\begin{eqnarray*}
\tau(\psi_1) &=& (\widetilde{\pi}u)^{q-1}(\psi_1-L(\chi_t,1)), \text{ and} \\
\varphi(\psi_1) &=& (\widetilde{\pi}u)^{1-q}\psi_1^q+L(\chi_t^q,1).
\end{eqnarray*}
Since $\varphi(\psi_1),L(\chi_t^q,1)$ both belong to $\mathcal{R}_1^0$, we
deduce that $\psi_1^qu^{1-q}$ belongs to $\mathcal{R}_1^0$.
More generally, the reader can check that
for all $d\geq 0$, and with $L=L(\chi_t,1)$ and $\mu=\widetilde{\pi}u$,
\begin{equation*}
\mu^{1-q^d}\psi^{q^{d}}+\mu^{1-q^{d-1}}\varphi(L)^{q^{d-1}}+\mu^{1-q^{d-2}}\varphi^2(L)^{q^{d-2}}+\cdots+\mu^{1-q}\varphi^{d-1}(L)^{q}\in\mathcal{R}_1^0.
\end{equation*}

One interesting consequence of the $\tau$ and $\varphi$ difference equations satisfied by $\psi_1$ is that by comparing the expansions in $\TT[[z]]$ of the left and right sides one may obtain recursion relations satisfied by the $L$-functions $L(\chi_t,n)$ with $n \equiv 1 \pmod{q-1}$.
\end{Example}

The structure of the algebras $\mathcal{R}_s^0$ is likely to be quite intricate.

\begin{Question}
Determine explicitly a set of generators of $\mathcal{R}_s^0$.
\end{Question}

Another natural question is the following.

\begin{Question}
Is the algebra $\mathcal{R}_s^0$ a $\tau$-stable algebra, that is,
for all $f\in \mathcal{R}_s^0$, $\tau(f)\in \mathcal{R}_s^0$?
\end{Question}
In this direction we observe, if $s=1$, that the sub-$\TT$-algebra
\[\mathcal{R}_1^{1} := \TT[\varphi^k(\psi_1);k\geq 0]\] 
of $\mathcal{R}_1^0$ is also $\tau$-stable. One may hope that $\mathcal{R}_1^{0}=\mathcal{R}_1^{1}$, but we are unable to show it presently.


\begin{thebibliography}{99}

\bibitem{GAtm} G. Anderson, {\em $t$-motives}, Duke Math. J. {\bf 53} (1986), 457--502.

\bibitem{AT} G. Anderson \& D. Thakur, {\em Tensor powers of the Carlitz module and zeta values.} Ann. of Math. (2) {\bf 132} (1990), no. 1, 159--191.

\bibitem{ANG&PEL1} B. Angl\`es \& F. Pellarin. {\em Functional identities for $L$-series values in positive characteristic.} J. Number Theory {\bf 142} (2014), 223--251.
 
 \bibitem{ANG&PEL2} B. Angl\`es \& F. Pellarin. {\em Universal Gauss-Thakur sums and $L$-series.} Invent. Math. (2014). DOI 10.1007/s00222-014-0546-8

\bibitem{APT} B. Angl\`es, F. Pellarin \& F. Tavares Ribeiro. {\em Arithmetic of positive characteristic L-series values in Tate algebras.} To Appear in Compos. Math. (2015). arXiv:1402.0120

\bibitem{Car35} L. Carlitz, \emph{On certain functions connected with polynomials in a Galois field.} Duke Math. J. {\bf 1} (1935), 137--168. 

\bibitem{Gekcrelles} E.-U. Gekeler, \emph{On the de Rham isomorphism for Drinfeld modules}. J. Reine Angew. Math. \textbf{401} (1989), 188--208.

\bibitem{GekelerCICMA} E.-U. Gekeler, {\em Lectures on Drinfeld Modular Forms,} CICMA Lecture Notes 1999--2004.

\bibitem{EGP} A. El-Guindy \& M. Papanikolas, \emph{Identities for Anderson generating functions for Drinfeld modules}, Monatsh. Math. \textbf{173} (2014), 471--493.
 
 \bibitem{GOS} D. Goss, {\em Basic Structures of Function Field Arithmetic.} Springer, Berlin, 1996.
 
 \bibitem{PAP} M. Papanikolas, {\em Tannakian duality for Anderson-Drinfeld motives and algebraic independence of Carlitz logarithms.}
Invent. Math. {\bf 171} (2008), no. 1, 123--174.

\bibitem{PEL0} F. Pellarin. {\em Aspects de l'ind\'ependance alg\'ebrique en caract\'eristique non nulle.}
S\'eminaire Bourbaki. Vol. 2006/2007.
Ast\'erisque No. 317 (2008), Exp. No. 973, viii, 205--242.

\bibitem{PEL} F. Pellarin. {\em Values of certain $L$-series in positive characteristic.} Ann. of Math. (2) {\bf 176} (2012), 2055--2093.

\bibitem{PEL1} F. Pellarin. {\em Interpolating Carlitz Zeta Values (Analytic Number Theory: Arithmetic Properties of Transcendental Functions and their Applications).} RIMS K\=oky\=uroku, {\bf 1898} (2014), 58--69 .

\bibitem{PER} R. Perkins. {\em Explicit formulae for $L$-values in positive characteristic.} {\bf 278} Math. Z. (2014), 279--299.
 
 \bibitem{SCH} W. H. Schikhof. {\em Ultrametric Calculus: An Introduction to $p$-Adic Analysis.}
 Cambridge University Press, (1984).

\bibitem{THA} D. Thakur. {\em Gauss sums for $\FF_q[t]$.}
Invent. Math. {\bf 94} (1988), 105--112.

\bibitem{THAsign} D. Thakur, \emph{Behaviour of function field Gauss sums at infinity}, Bull. London Math. Soc. \textbf{25} (1993), 417--426.
\end{thebibliography}
\end{document}